\documentclass[11pt,leqno]{article}
\usepackage{exscale,relsize}
\usepackage{amsmath}
\usepackage{amsfonts}
\usepackage{amssymb}
\usepackage{calc}
\usepackage{theorem}
\usepackage{pifont}      
\newcommand{\scal}[2]{\langle{{#1},{#2}}\rangle}

\newcommand{\RR}{\ensuremath{\mathbb R}}

\newcommand{\RX}{\ensuremath{\,\left]-\infty,+\infty\right]}}
\newcommand{\RXX}{\ensuremath{\,\left[-\infty,+\infty\right]}}

\newcommand{\thalb}{\ensuremath{\tfrac{1}{2}}}

\newcommand{\menge}[2]{\big\{{#1} \mid {#2}\big\}}

\newcommand{\To}{\ensuremath{\rightrightarrows}}

\newcommand{\dom}{\ensuremath{\operatorname{dom}}}

\newcommand{\gra}{\ensuremath{\operatorname{gra}}}

\newcommand{\Id}{\ensuremath{\operatorname{Id}}}

\renewcommand{\phi}{\ensuremath{\varphi}}

\newtheorem{theorem}{Theorem}[section]

\newtheorem{fact}[theorem]{Fact}
\newtheorem{corollary}[theorem]{Corollary}
\newtheorem{proposition}[theorem]{Proposition}
\newtheorem{definition}[theorem]{Definition}

\theoremstyle{plain}{\theorembodyfont{\rmfamily}
}
\theoremstyle{plain}{\theorembodyfont{\rmfamily}
}
\theoremstyle{plain}{\theorembodyfont{\rmfamily}
}
\theoremstyle{plain}{\theorembodyfont{\rmfamily}
}
\theoremstyle{plain}{\theorembodyfont{\rmfamily}
\newtheorem{remark}[theorem]{Remark}}
\theoremstyle{plain}{\theorembodyfont{\rmfamily}
}

\begin{document}


\title{\textsc{The Br\'{e}zis-Browder Theorem revisited and properties of Fitzpatrick functions of order $n$}}

\author{
Liangjin\ Yao\thanks{Mathematics, Irving K.\ Barber School, UBC Okanagan,
Kelowna, British Columbia V1V 1V7, Canada.
E-mail:  \texttt{ljinyao@interchange.ubc.ca}.}}
 \vskip 3mm

\date{May 22, 2009}
\maketitle

\begin{abstract} \noindent In this note, we  study maximal monotonicity of  linear relations
 (set-valued operators with linear graphs) on  reflexive Banach spaces.
 We provide a new and simpler proof
of a result due to Br\'{e}zis-Browder which states that
   a monotone linear relation with closed graph is maximal monotone
    if and only if its adjoint is monotone. We also study Fitzpatrick functions
     and give an explicit formula for Fitzpatrick functions of order $n$
     for monotone symmetric linear relations.
\end{abstract}



\section{Introduction}
Monotone operators play important roles in convex analysis and optimization
\cite{BurIus,ph,Si,Si2,RockWets,Zalinescu,Zeidler}.
 In 1978, Br\'{e}zis-Browder gave some characterizations of a monotone
operator with closed graph (\cite[Theorem~2]{Brezis-Browder}). The Br\'{e}zis-Browder Theorem states
that a
 monotone
linear relation with closed graph is maximal monotone if and only if its adjoint is monotone
 if and only if its adjoint is maximal monotone,
which gives the connection  between the monotonicity of a linear relation and that of its adjoint.
Now we give a new and simpler proof
of the hard part of the Br\'{e}zis-Browder Theorem (Theorem~\ref{Sv:2}):  a
 monotone
linear relation with closed graph is maximal monotone if its adjoint is monotone.

We suppose throughout this note that
$X$ is  a real reflexive Banach space with norm $\|\cdot\|$,
 that $X^*$ is its continuous dual space with norm
$\|\cdot\|_*$, and dual product $\scal{\cdot}{\cdot}$. We now introduce some notation.
Let $A\colon X\To X^*$
be a \emph{set-valued operator} or \emph{multifunction} whose graph is defined by
\begin{align*}\gra A:= \menge{(x,x^*)\in X\times X^*}{x^*\in Ax}.\end{align*}
The \emph{inverse
operator}  of $A$, $A^{-1}\colon X^*\To X$, is  given by $\gra A^{-1} :=
\menge{(x^*,x)\in X^*\times X}{x^*\in Ax}$;
 the \emph{domain} of $A$ is $\dom A := \menge{x\in X}{Ax\neq\varnothing}$.

If $Z$ is a real reflexive Banach space with dual $Z^*$ and a set $S\subseteq Z$, we denote $S^\bot$ by
$S^\bot := \menge{z^*\in Z^*}{\langle z^*,s\rangle= 0,\quad \forall s\in S}$.
Then the \emph{adjoint} of $A$, denoted by $A^*$, is defined
by
\begin{equation*}
\gra A^* :=
\menge{(x,x^*)\in X\times X^*}{(x^*,-x)\in (\gra A)^\bot}.
\end{equation*}
Note that $A$ is said to be  a \emph{linear relation}
if $\gra A$ is a linear subspace of
$X\times X^*$. (See \cite{Cross}
 for further information on linear
relations.)
Recall that $A$ is said to be \emph{monotone} if for all $(x,x^*), (y,y^*)\in\gra
A$ we have
\begin{equation*}
\scal{x-y}{x^*-y^*}\geq 0,
\end{equation*}and $A$ is \emph{maximal monotone} if $A$ is monotone and
 $A$ has no proper monotone extension.
We say $(x,x^*)\in X\times X^*$ is \emph{monotonically related} to $\gra A$ if
(for every $(y,y^*)\in\gra A)$
$\scal{x-y}{x^*-y^*}\geq 0$.
Recently
linear relations have been become an interesting object and comprehensively studied
 in Monotone Operator Theory:
see \cite{BBBRW,BB,BBW,BWY2,BWY3,BWY4,PheSim,Svaiter,Voisei06,VZ}.
We can now precisely describe the Br\'{e}zis-Browder Theorem.
Let $A$ be a monotone linear relation with closed graph. Then
\begin{align*} A\;\text{is maximal monotone}&\Leftrightarrow A^*\;\text{ is 
maximal monotone}\\
&\Leftrightarrow A^*\;\text{ is monotone}.\end{align*}
Our goal of this paper is to give  a simpler proof of Br\'{e}zis-Browder Theorem
and to derive more properties of Fitzpatrick functions of order $n$.
The paper is organized as follows.
The first main result (Theorem~\ref{Sv:2}) is proved
in Section~\ref{s:main} providing a new and simpler proof of the
 Br\'{e}zis-Browder  Theorem.
 In Section~\ref{main:2}, some explicit formula
 for Fitzpatrick functions are given. Recently,
  \emph{Fitzpatrick functions of order $n$}~\cite{BBBRW} have turned out to be a useful
  tool
in the study of $n$-cyclic monotonicity
 (see \cite{BBBRW, BLW, BBW}). Theorem~\ref{Festival:2}
  gives an explicit formula for Fitzpatrick functions of order $n$ associated with
 symmetric linear relations, which generalizes and simplifies
 \cite[Example 4.4]{BBBRW} and \cite[Example 6.4]{BBW}.

Our notation is standard. The notation
$A\colon X\rightarrow X^*$ means that $A$ is a \emph{single-valued} mapping (with full domain)
from $X$ to $X^*$.
Given a subset $C$ of $X$, $\overline{C}$ is the closure of $C$.
The \emph{indicator function} $\iota_C:X\rightarrow\RX$ of $C$ is defined by
\begin{equation}x\mapsto\begin{cases}0,\;&\text{if}\;x\in C;\\
+\infty,\;&\text{otherwise}.\end{cases}\end{equation}
For a function  $f\colon X\to \RX$,
$\dom f= \{x\in X \mid f(x)<+\infty\}$ and
$f^*\colon X^*\to\RXX\colon x^*\mapsto\sup_{x\in X}(\scal{x}{x^*}-f(x))$ is
the \emph{Fenchel conjugate} of $f$.
Recall that $f$ is said to be
  proper if $\dom f\neq\varnothing$. If $f$ is convex,
   $\partial f\colon X\rightrightarrows X^*\colon x\mapsto \menge{x^*\in X^*}{(\forall y\in
X)\; \scal{y-x}{x^*} + f(x)\leq f(y)}$
is the
 \emph{subdifferential operator} of $f$.
Denote ${J}$ by the duality map, i.e.,
the subdifferential of the function $\tfrac{1}{2}\|\cdot\|^2$, by \cite[Example~2.26]{ph},
\begin{align*}
Jx:=\{x^*\in X^*\mid\langle x^*,x\rangle=\|x^*\|_*\cdot\|x\|,\;\text{with $\|x^*\|_*=\|x\|$}\}.\end{align*}

\section{New proof of the Br\'{e}zis-Browder Theorem}
\label{s:main}

\begin{fact}[Simons]
\emph{(See \cite[Lemma~19.7 and Section~22]{Si2}.)}
\label{f:referee}
Let $A:X\To X^*$ be a monotone linear relation such that $\gra A
\neq\varnothing$.
Then the function
\begin{equation}
g\colon X\times X^* \to \RX\colon
(x,x^*)\mapsto \scal{x}{x^*} + \iota_{\gra A}(x,x^*)
\end{equation}
is proper and convex.
\end{fact}

\begin{fact}[Simons-Z{\u{a}}linescu]\emph{(See \cite[Theorem~1.2]{SimZal}.)}\label{SV:1}
 Let  $A\colon X \rightrightarrows X^*$ be monotone. Then $A$ is maximal monotone
  if and only if $$\gra A+\gra (-{J})=X\times X^*.$$
\end{fact}

\begin{remark} When $J$ and $J^{-1}$ are single-valued, Fact~\ref{SV:1} yields
Rockafellar's characterization of maximal monotonicity of $A$. See \cite[Theorem~1.3]{SimZal}
and \cite[Theorem~29.5 and Remark~29.7]{Si2}.
\end{remark}

Now we state the Br\'{e}zis-Browder Theorem.
\begin{theorem}[Br\'ezis-Browder] \emph{(See \cite[Theorem~2]{Brezis-Browder}.)}
\label{Sv:7}
Let  $A\colon X \rightrightarrows X^*$ be a linear relation with closed graph.
 Then the following statements are equivalent.
\begin{enumerate}
 \item $A$ is maximal monotone.
\item  $A^*$ is maximal  monotone.
\item   $A^*$ is monotone.
\end{enumerate}
\end{theorem}
\begin{proof}
(i)$\Rightarrow$(iii): Suppose to the contrary that $A^*$ is not monotone.
Then there exists $(x_0,x_0^*)\in\gra A^*$ such that $\langle x_0,x_0^*\rangle<0$.
Now we have
\begin{align}&\langle-x_0-y,x^*_0-y^*\rangle
=\langle-x_0,x^*_0\rangle+\langle y,y^*\rangle+\langle x_0,y^*\rangle
+\langle -y,x_0^*\rangle\nonumber\\
&=\langle-x_0,x^*_0\rangle+\langle y,y^*\rangle>0,
\quad \forall (y,y^*)\in\gra A.\label{Brezi:1}\end{align}
Thus, $(-x_0,x^*_0)$ is monotonically related to $\gra A$.
 By maximal monotonicity of $A$, $(-x_0,x^*_0)\in\gra A$.
Then $\langle-x_0-(-x_0),x^*_0-x_0^*\rangle=0$,
 which contradicts \eqref{Brezi:1}. Hence $A^*$ is monotone.

The hard parter is to show (iii)$\Rightarrow$(i). See Theorem~\ref{Sv:2} below.

(i)$\Leftrightarrow$(ii): Apply directly (iii)$\Leftrightarrow$(i)
 by using $A^{**}=A$ (since $\gra A$ is closed).
\end{proof}

In Theorem~\ref{Sv:2}, we provide a new and simpler proof to
 show the hard part (iii)$\Rightarrow$(i) in  Theorem~\ref{Sv:7}.
  The proof  was inspired by that of \cite[Theorem~32.L]{Zeidler}.

\begin{theorem}\label{Sv:2}
Let  $A\colon X \rightrightarrows X^*$ be a  monotone linear
 relation with closed graph. Suppose $A^*$ is monotone. Then $A$ is maximal monotone.
\end{theorem}

\begin{proof}
 We show that $X\times X^*\subseteq\gra A+\gra (-{J})$.
   Let $(x,x^*)\in X\times X^*$ and we define $g:X\times X^*\rightarrow\RX$ by
\begin{align*}(y,y^*)\mapsto\tfrac{1}{2}\|y^*\|^2_*
+\tfrac{1}{2}\| y\|^2+\langle y^*,y\rangle+\iota_{\gra A}(y-x,y^*-x^*).
\end{align*}
Since $\gra A$ is closed, $g$ is lower semicontinuous on $X\times X^*$.
By Fact~\ref{f:referee}, $g$ is  convex and coercive.
 Here $g$ has minimizer. Suppose that $(z,z^*)$ is a minimizer of $g$.
Then $(z-x,z^*-x^*)\in\gra A$, that is,
\begin{align}(x,x^*)\in\gra A+(z,z^*).\label{SL:1}\end{align}
On the other hand,  since $(z,z^*)$ is a minimizer of $g$, $(0,0)\in\partial g(z,z^*)$.
By a result of Rockafellar (see \cite[Theorem~2.9.8]{Clarke}),
 there exist $(z_0^*, z_0)
\in\partial (\iota_{\gra A}(\cdot-x,\cdot-x^*))(z,z^*)=
\partial \iota_{\gra A}(z-x,z^*-x^*)
=(\gra A)^{\bot}$, and $(v,v^*)\in X\times X^*$  with $v^*\in {J}z, z^*\in {J}v$ such that
$$(0,0)=(z^*,z)+(v^*,v)                                         +(z_0^*, z_0).$$
Then
\begin{align*}\big(-(z+v), z^*+v^*\big)\in\gra A^{*}.\end{align*}
Since $A^*$ is monotone,
\begin{align}\langle z^*+v^*,z+v\rangle=\langle z^*, z\rangle
+\langle z^*,v\rangle+\langle v^*,z\rangle+\langle v^*,v\rangle\leq 0.\label{SS:2}\end{align}
Note that since $\langle z^*,v\rangle=\|z^*\|_*^2=\|v\|^2,\;\langle v^*,z\rangle=\|v^*\|_*^2=\|z\|^2$,
 by \eqref{SS:2}, we have
$$\tfrac{1}{2}\|z\|^2+\tfrac{1}{2}\|z^*\|_*^2+\langle z^*, z\rangle
+\tfrac{1}{2}\|v^*\|^2_*+\tfrac{1}{2}\|v\|^2+\langle v,v^*\rangle\leq 0.$$
Hence $z^*\in-{J}z$. By \eqref{SL:1}, $(x,x^*)\in\gra A+\gra (-{J})$.
Thus, $X\times X^*\subseteq \gra (-{J})+\gra A$.
   By Fact~\ref{SV:1}, $A$ is maximal monotone.
\end{proof}

\section{Fitzpatrick functions and Fitzpatrick functions of order $n$}\label{main:2}
Now we introduce some  properties of monotone linear relations.

\begin{fact}\emph{(See \cite{BWY3}.)}\label{linear}
Assume that $A:X\To X^*$ is a monotone linear relation. Then
the following hold.
\begin{enumerate}
\item\label{Nov:s2}
The function
$\dom A\rightarrow\RR: y\mapsto \langle y, Ay\rangle$
is convex.
\item\label{Sia:2c} $\dom A \subseteq (A0)^{\perp}$.  For every $x\in (A0)^{\perp}$, the function
$\dom A\rightarrow\RR: y\mapsto \langle x, Ay\rangle$
is linear.

\end{enumerate}
\end{fact}
\begin{proof}
\ref{Nov:s2}: See \cite[Proposition~2.3]{BWY3}.
\ref{Sia:2c}: See \cite[Proposition~2.2(i)(iii)]{BWY3}.
\end{proof}

\begin{definition}
Suppose $A:X\rightrightarrows X^*$ is a monotone linear relation.
 We say $A$ is \emph{symmetric} if $\langle Ax,y\rangle=\langle Ay,x\rangle,\quad
\forall x,y\in\dom A.$\end{definition}

For a monotone linear relation $ A\colon X \To X^*$ it will be convenient
to define (as in, e.g., \cite{BBW})
\begin{equation} \label{e:diequad}
q_A \colon X\to\RR\colon x\mapsto\begin{cases} \thalb \scal{x}{Ax},&\text{if}\; x\in\dom A;\\
\infty,&\text{otherwise}.\end{cases}
\end{equation}
By Fact~\ref{linear}\ref{Nov:s2}, $q_A$ is at most single-valued and convex.

The following generalizes a result of Phelps-Simons
(see \cite[Theorem 5.1] {PheSim}) from symmetric monotone linear operators
 to symmetric monotone linear
relations. We write $\overline{f}$ for the lower semicontinuous hull of $f$.
\begin{proposition} \label{f:PheSim} Let $A:X\rightrightarrows X^*$
 be a monotone symmetric linear relation.
Then
\begin{enumerate}
\item \label{f:PheSim:lsc02} $q_{A}$ is convex, and $\overline{q_A}+\iota_{\dom A}=q_A$.
\item \label{f:PheSim:quad}
$\gra A \subseteq\gra \partial \overline{q_A}$. If $A$ is maximal monotone, then $A=\partial \overline{q_A}$.
\end{enumerate}
\end{proposition}
\begin{proof}
Let $x\in\dom A$.

\ref{f:PheSim:lsc02}: Since $A$ is monotone, $q_{A}$ is convex. Let $y\in\dom A$.
Since  $A$ is monotone, by Fact~\ref{linear}\ref{Sia:2c},
\begin{align}0\leq\tfrac{1}{2}\langle Ax-Ay,x-y\rangle=
\tfrac{1}{2}\langle Ay,y\rangle+\tfrac{1}{2}\langle Ax,x\rangle-\langle Ax, y\rangle,
\label{Bor:u2}\end{align}
we have $q_A(y)\geq\langle Ax, y\rangle-q_A(x)$.
 Take lower semicontinuous hull and then deduce that $\overline{q_A}(y)\geq\langle Ax, y\rangle-q_A(x)$.
  For $y=x$, we have $\overline{q_A}(x)\geq q_A(x)$. On the other hand,
  $\overline{q_A}(x)\leq q_A(x)$.
Altogether, $\overline{q_A}(x)=q_A(x)$.
Thus \ref{f:PheSim:lsc02} holds.

\ref{f:PheSim:quad}:
Let $y\in\dom A$. By \eqref{Bor:u2} and \ref{f:PheSim:lsc02},
\begin{align}q_A(y)\geq
q_A(x)+\langle Ax,y-x\rangle=\overline{q_A}(x)+\langle Ax,y-x\rangle.\label{Bor:u3}
\end{align}
Since $\dom\overline{q_A}\subseteq\overline{\dom q_A}=\overline{\dom A}$, by \eqref{Bor:u3},
 $\overline{q_A}(z)\geq
\overline{q_A}(x)+\langle Ax,z-x\rangle,\quad \forall z\in\dom\overline{q_A}.$
Hence $Ax\subseteq\partial \overline{q_A}(x)$.
If $A$ is maximal monotone, $A=\partial \overline{q_A}$.
Thus \ref{f:PheSim:quad} holds.
\end{proof}

\begin{definition}
Let $A\colon X\rightrightarrows X^*.$ The
\emph{Fitzpatrick function} of $A$ is
\begin{equation}
F_A\colon  (x,x^*)\mapsto \sup_{(a,a^*)\in
\gra A}\big(\langle x, a^*\rangle+ \langle a,x^*\rangle-\langle a,a^*\rangle\big).
\end{equation}
\end{definition}

\begin{definition}[Fitzpatrick family]
Let  $A\colon X \To X^*$ be a maximal monotone operator.
The associated \emph{Fitzpatrick family}
$\mathcal{F}_A$ consists of all functions $F\colon
X\times X^*\to\RX$ that are lower semicontinuous and convex,
and that satisfy
$F\geq \scal{\cdot}{\cdot} $, and $F=\scal{\cdot}{\cdot}$ on $\gra A$.
\end{definition}

Following \cite{Penot2}, it will be convenient to set
$ F^\intercal\colon X^*\times X\colon (x^*,x)\mapsto F(x,x^*)$,
when
$F\colon X\times X^*\to\RX$, and similarly
for a function defined on $X^*\times X$.

\begin{fact}[Fitzpatrick]\label{GF:2}
\emph{(See \cite[Theorem~3.10]{Fitz88} or \cite[Corollary~4.1]{burachick}.)}
Let  $A\colon X \To X^*$ be a maximal monotone operator.
Then for every $(x,x^*)\in X\times X^*$,
\begin{equation}
F_A(x,x^*) = \min\menge{F(x,x^*)}{F\in \mathcal{F}_A}
\quad\text{and}\quad
F_A^{*\intercal}(x,x^*) = \max\menge{F(x,x^*)}{F\in \mathcal{F}_A}.
\end{equation}
\end{fact}

\begin{proposition}\label{better}
Let $A\colon X\rightrightarrows X^*$ be a maximal monotone and symmetric linear relation. Then
\begin{align*}
F_A(x,x^*)=\tfrac{1}{2}\overline{q_A}(x)+\tfrac{1}{2}\langle x,x^*\rangle
+\tfrac{1}{2}q^*_A(x^*),
\quad \forall (x,x^*)\in X\times X^*.
\end{align*}
\end{proposition}

\begin{proof}Define function $k:X\times X^*\rightarrow\RX$ by
\begin{align*}(z,z^*)\mapsto \tfrac{1}{2}\overline{q_A}(z)+\tfrac{1}{2}\langle z,z^*\rangle
+\tfrac{1}{2}q^*_A(z^*).
\end{align*}
Claim 1: $F_A=k$ on $\dom A\times X^*$.

Let $(x,x^*)\in X\times X^*$, and
suppose that $x\in\dom A$. Then
\begin{align*}
F_A(x,x^*)&=\sup_{(y,y^*)\in
\gra A}\Big(\langle x, y^*\rangle+ \langle y,x^*\rangle-\langle y,y^*\rangle\Big)\\
&=\sup_{y\in\dom A}\Big(\langle x, Ay\rangle+ \langle y,x^*\rangle-2q_A(y)\Big)\\
&=\tfrac{1}{2}\ q_A(x)+\sup_{y\in\dom A}\Big(\langle Ax, y\rangle
+ \langle y,x^*\rangle-\tfrac{1}{2}\ q_A(x)-2q_A(y)\Big)\\
&=\tfrac{1}{2}\ q_A(x)+\tfrac{1}{2}\sup_{y\in\dom A}\Big(\langle Ax, 2y\rangle
+ \langle 2y,x^*\rangle-\ q_A(x)-4q_A(y)\Big)\\
&=\tfrac{1}{2}\ q_A(x)+\tfrac{1}{2}\sup_{z\in\dom A}\Big(\langle Ax, z\rangle
+ \langle z,x^*\rangle-\ q_A(x)-q_A(z)\Big)\\
&=\tfrac{1}{2}\ q_A(x)+\tfrac{1}{2}\sup_{z\in\dom A}\Big(\langle z,x^*\rangle-\ q_A(z-x)\Big)\\
&=\tfrac{1}{2}\ q_A(x)+\tfrac{1}{2}\langle x,x^*\rangle
+\tfrac{1}{2}\sup_{z\in\dom A}\Big(\langle z-x,x^*\rangle-\ q_A(z-x)\Big)\\
&=\tfrac{1}{2} q_A(x)+\tfrac{1}{2}\langle x,x^*\rangle+\tfrac{1}{2}q^*_A(x^*)\\
&=k(x,x^*)\quad(\text{by Proposition~\ref{f:PheSim}(i)}).
\end{align*}
Claim 2: $k$ is  convex and proper lower semicontinuous on $X\times X^*$.

Since $F_A$ is convex, $\tfrac{1}{2} q_A+\tfrac{1}{2}\langle \cdot,\cdot\rangle+\tfrac{1}{2}q^*_A$
is convex on $\dom A\times X^*$. Now we show that $k$ is convex. Let $\{(a,a^*),(b,b^*)\}\subseteq\dom k$,
and $t\in\left]0,1\right[$.
Then we have $\{a,b\}\subseteq\dom\overline{q_A}\subseteq\overline{\dom A}$. Thus, there exist $(a_n), (b_n)$ in $\dom A$
such that $a_n\rightarrow a, b_n\rightarrow b$ with $q_A(a_n)\rightarrow\overline{q_A}(a),
q_A(b_n)\rightarrow\overline{q_A}(b)$.
Since $\tfrac{1}{2} q_A+\tfrac{1}{2}\langle \cdot,\cdot\rangle+\tfrac{1}{2}q^*_A$
is convex on $\dom A\times X^*$, we have
\begin{align}
&\big(\tfrac{1}{2} q_A+\tfrac{1}{2}\langle \cdot,\cdot\rangle+\tfrac{1}{2}q^*_A\big)\big(ta_n+(1-t)b_n,ta^*+(1-t)b^*\big)
\nonumber\\&
\leq t\big(\tfrac{1}{2} q_A+\tfrac{1}{2}\langle \cdot,\cdot\rangle+\tfrac{1}{2}q^*_A\big)(a_n,a^*)
+(1-t)\big(\tfrac{1}{2} q_A+\tfrac{1}{2}\langle \cdot,\cdot\rangle+\tfrac{1}{2}q^*_A\big)(b_n,b^*).\label{Sff:2}\end{align}
Take $\liminf$ on both sides of \eqref{Sff:2} to see that
\begin{align*}
k\big(ta+(1-t)b,ta^*+(1-t)b^*\big)\leq tk(a,a^*)+(1-t)k(b,b^*).\end{align*}
Hence $k$ is convex on $X\times X^*$.
Thus, $k$ is convex and proper lower semicontinuous.

Claim 3: $F_A=k$ on $X\times X^*$. To this end, we first observe that
\begin{align}\dom\partial k^*=\gra A^{-1}.\label{Bfi:2}\end{align}
We have
\begin{align*}&(w^*,w)\in\dom\partial k^*\Leftrightarrow(w^*,w)
\in\dom\partial (2k)^*\Leftrightarrow (a,a^*)\in\partial (2k)^*(w^*,w),
\quad \exists (a,a^*)\in X\times X^*\\
&\Leftrightarrow (w^*,w)\in\partial(2k)(a,a^*)
 \Leftrightarrow (w^*-a^*,w-a)\in\partial (\overline{q_A}\small\oplus q^*_A)(a,a^*),
 \quad(\text{by  \cite[Theorem~2.9.8]{Clarke}})\\
&\Leftrightarrow w^*-a^*\in\partial \overline{q_A}(a),\; w-a\in\partial q^*_A(a^*)\\
&\Leftrightarrow w^*-a^*\in \partial \overline{q_A}(a),\; a^*\in\partial \overline{q_A}(w-a)\\
&\Leftrightarrow w^*-a^*\in Aa,\; a^*\in A(w-a),\quad(\text{by  Proposition~\ref{f:PheSim}(ii)})\\
&\Leftrightarrow (w,w^*)\in\gra A\Leftrightarrow (w^*,w)\in\gra A^{-1}.\end{align*}
Next we observe that
\begin{align}k^{*\intercal}(z,z^*)=\langle z,z^*\rangle,\quad \forall(z,z^*)\in\gra A.\label{Bfi:1}\end{align}

Since $k(z,z^*)\geq\langle z,z^*\rangle$
and
\begin{align*}k(z,z^*)=\langle z,z^*\rangle\Leftrightarrow
 \overline{q_A}(z)+q^*_A(z^*)=\langle z,z^*\rangle
\Leftrightarrow z^*\in\partial\overline{q_A}(z)
=Az\quad(\text{by Proposition~\ref{f:PheSim}(ii)}),\end{align*}
Fact~\ref{GF:2} implies that
$F_A\leq k\leq F^{*\intercal}_A$. Hence
$F_A\leq k^{*\intercal}\leq F_A^{*\intercal}$. Then by Fact~\ref{GF:2},
\eqref{Bfi:1} holds.

Now using \eqref{Bfi:1}\eqref{Bfi:2} and
a result by J.\ Borwein (see \cite[Theorem~1]{Borwein82} or
\cite[Theorem~3.1.4(i)]{Zalinescu}), we have
$k=k^{**}=(k^*+\iota_{\dom\partial k^*})^*
=(\langle\cdot,\cdot\rangle+\iota_{\gra A^{-1}})^*=F_A.$
\end{proof}

\begin{definition}[Fitzpatrick functions of order $n$]\emph{\cite[Definition~2.2 and Proposition~2.3]{BBBRW}}
Let $A: X\rightrightarrows X^*$. For every $n\in\{2,3,\ldots\}$,
 the \emph{Fitzpatrick function of $A$ of order $n$} is
\begin{equation*}F_{A,\, n}(x,x^*):=
\sup\limits_{\big\{(a_1,a^*_1),\cdots (a_{n-1},a^*_{n-1})\big\}\subseteq\gra A }
\Big(\langle x,x^*\rangle+\Big(\sum^{n-2}_{i=1}\langle a_{i+1}-a_i, a^*_i\rangle\Big)
+\langle x-a_{n-1},a^*_{n-1}\rangle+\langle a_1-x,x^*\rangle\Big).\end{equation*}
Clearly, $F_{A,\, 2}=F_A$.
We set $F_{A,\, \infty}=\sup_{n\in\{2,3,\cdots\}}F_{A,\, n}$.

\end{definition}
\begin{fact}[recursion]\label{Ind:1}\emph{(See {\cite[Proposition~2.13]{BLW}}.)}
Let $A: X\rightrightarrows X^*$  be  monotone, and let $n\in\{2,3,\ldots\}$. Then
\begin{equation*}F_{A,\, n+1}(x,x^*)=
\sup_{(a,a^*)\in\gra A}\big(F_{A,\, n}(a,x^*)+\langle x-a,a^*\rangle\big),
\quad \forall (x,x^*)\in X\times X^*.
\end{equation*}

\end{fact}

\begin{theorem}\label{Festival:2}
Let $A: X\rightrightarrows X^*$  be a maximal monotone and symmetric linear relation,
let $n\in\{2,3,\ldots\}$, and let $(x,x^*)\in X\times X^*$. Then
\begin{align}F_{A,\, n}(x,x^*)=\tfrac{n-1}{n}\overline{q_A}(x)+\tfrac{n-1}{n}q^*_A(x^*)
+\tfrac{1}{n}\langle x,x^*\rangle,
\label{New:1}\end{align}consequently, $F_{A,\, n}(x,x^*)=\tfrac{2(n-1)}{n}F_A(x,x^*)+\tfrac{2-n}{n}\langle x,x^*\rangle$.
Moreover, \begin{align}F_{A,\, \infty}=\overline{q_A}\small\oplus
 q^*_A=2F_A-\langle\cdot,\cdot\rangle.\label{New:2}\end{align}

\end{theorem}

\begin{proof}Let $(x,x^*)\in X\times X^*$.
The proof is by induction on $n$. If $n=2$, then the result follows for Proposition~\ref{better}.

Now assume that \eqref{New:1} holds for $n\geq2$.
Using Fact~\ref{Ind:1}, we see that
\begin{align*}
&F_{A,\, n+1}(x,x^*)=\sup_{(a,a^*)\in\gra A}\Big(F_{A,\, n}(a,x^*)
+\langle x-a,a^*\rangle\Big)\\
&=\sup_{(a,a^*)\in\gra A}\Big(\tfrac{n-1}{n}q^*_A(x^*)
+\tfrac{n-1}{n}\overline{q_A}(a)+\tfrac{1}{n}\langle a,x^*\rangle+\langle x-a,a^*\rangle\Big)\\
&=\tfrac{n-1}{n}q^*_A(x^*)+\sup_{(a,a^*)\in\gra A}\Big(\tfrac{n-1}{2n}\langle a,a^*\rangle
+\langle a,\tfrac{1}{n}x^*\rangle+\langle x,a^*\rangle-\langle a,a^*\rangle \Big),
\quad(\text{by  Proposition~\ref{f:PheSim}(i)})\\
&=\tfrac{n-1}{n}q^*_A(x^*)+\sup_{(a,a^*)\in\gra A}\Big(
\langle a,\tfrac{1}{n}x^*\rangle+\langle x,a^*\rangle-
\tfrac{n+1}{2n}\langle a,a^*\rangle \Big)\\
&=\tfrac{n-1}{n}q^*_A(x^*)+\tfrac{2n}{n+1}\sup_{(a,a^*)\in\gra A}\Big(
\langle \tfrac{n+1}{2n}a,\tfrac{1}{n}x^*\rangle+\langle x,\tfrac{n+1}{2n}a^*\rangle-
\langle \tfrac{n+1}{2n}a,\tfrac{n+1}{2n}a^*\rangle \Big)\\
&=\tfrac{n-1}{n}q^*_A(x^*)+\tfrac{2n}{n+1}\sup_{(b,b^*)\in\gra A}\Big(
\langle b,\tfrac{1}{n}x^*\rangle+\langle x,b^*\rangle-
\langle b,b^*\rangle \Big)\\
&=\tfrac{n-1}{n}q^*_A(x^*)+\tfrac{2n}{n+1}F_A(x,\tfrac{1}{n}x^*)
\\
&=\tfrac{n-1}{n}q^*_A(x^*)+\tfrac{n}{n+1}q^*_A(\tfrac{1}{n}x^*)+\tfrac{n}{n+1}\overline{q_A}(x)
+\tfrac{1}{n+1}\langle x^*,x\rangle\quad(\text{by Proposition~\ref{better}})\\
&=\tfrac{n-1}{n}q^*_A(x^*)+\tfrac{1}{(n+1)n}q^*_A(x^*)+\tfrac{n}{n+1}
\overline{q_A}(x)+\tfrac{1}{n+1}\langle x^*,x\rangle\\
&=\tfrac{n}{n+1}q^*_A(x^*)+\tfrac{n}{n+1}\overline{q_A}(x)
+\tfrac{1}{n+1}\langle x,x^*\rangle,\end{align*}
which is the result for $n+1$.
 Thus,
 by Proposition~\ref{better}, $F_{A,\, n}(x,x^*)
=\tfrac{2(n-1)}{n}F_A(x,x^*)+\tfrac{2-n}{n}\langle x,x^*\rangle$.\\
By \eqref{New:1}, $\dom F_{A,\, n}=\dom(\overline{q_A}\small\oplus q^*_A)$.
Now suppose that $(x,x^*)\in
\dom F_{A,\, n}$.\\
By  $\overline{q_A}(x)+ q^*_A(x^*)-F_{A,\, n}(x,x^*)
=\tfrac{1}{n}\Big(\overline{q_A}(x)+q^*_A(x^*)-\langle x,x^*\rangle\Big)\geq 0$ and
$$F_{A,\, n}(x,x^*)\rightarrow (\overline{q_A}\small\oplus q^*_A)(x,x^*),\;n\rightarrow\infty.$$
 Thus, \eqref{New:2} holds.
\end{proof}

\begin{remark}Theorem~\ref{Festival:2} generalizes and simplifies
 \cite[Example 4.4]{BBBRW} and \cite[Example 6.4]{BBW}. See Corollary~\ref{Cof:1}.\end{remark}
\begin{remark} Formula Identity
\eqref{New:1} does not hold for nonsymmetric linear relations. See \cite[Example~2.8]{BBW}
for an example when $A$ is skew linear operator and \eqref{New:1} fails.
\end{remark}

\begin{corollary}\label{Cof:1}
Let $A: X\rightarrow X^*$  be a maximal monotone and symmetric linear operator,
let $n\in\{2,3,\ldots\}$, and let $(x,x^*)\in X\times X^*$. Then
\begin{align}F_{A,\, n}(x,x^*)=\tfrac{n-1}{n}{q_A}(x)+\tfrac{n-1}{n}q^*_A(x^*)
+\tfrac{1}{n}\langle x,x^*\rangle,
\end{align}
and, \begin{align}F_{A,\, \infty}=q_A\small\oplus q^*_A.\end{align}
If $X$ is a Hilbert space, then
\begin{align}F_{\Id,\, n}(x,x^*)=\tfrac{n-1}{2n}\|x\|^2+\tfrac{n-1}{2n}\|x^*\|^2
+\tfrac{1}{n}\langle x,x^*\rangle,
\end{align}
and, \begin{align}F_{\Id,\, \infty}=\tfrac{1}{2}\|\cdot\|^2\small\oplus \tfrac{1}{2}\|\cdot\|^2.\end{align}
\end{corollary}

\begin{definition}
Let $F_1, F_2\colon X\times X^*\rightarrow\RX$.
Then the \emph{partial inf-convolution} $F_1\Box_2 F_2$
is the function defined on $X\times X^*$ by
\begin{equation*}F_1\Box_2 F_2\colon
(x,x^*)\mapsto \inf_{y^*\in X^*}\big(
F_1(x,x^*-y^*)+F_2(x,y^*)\big).
\end{equation*}
\end{definition}

\begin{theorem}[nth order Fitzpatrick function of the sum]\label{FS6}
Let $A,B\colon X\To X^*$ be maximal monotone and symmetric linear relations,
 and let $n\in\{2,3,\cdots\}$.
Suppose that $\dom A-\dom B$ is closed.
Then $F_{A+B,\, n}= F_{A,\,n}\Box_2F_{B,\,n}$. Moreover,
$F_{A+B,\, \infty}= F_{A,\,\infty}\Box_2F_{B,\,\infty}$.
\end{theorem}
\begin{proof}By \cite[Theorem~5.5]{SiZ} or \cite{Voisei06}, $A+B$ is maximal monotone.
Hence $A+B$ is a maximal monotone and symmetric linear relation.
Let $(x,x^*)\in X\times X^*$. Then by Theorem~\ref{Festival:2},
\begin{align*}
&F_{A,\,n}\Box_2F_{B,\,n}(x,x^*)\\&=\inf_{y^*\in X^*}
\Big(\tfrac{2(n-1)}{n}F_A(x,y^*)+\tfrac{2-n}{n}\langle x,y^*\rangle+
\tfrac{2(n-1)}{n}F_B(x,x^*-y^*)+\tfrac{2-n}{n}\langle x,x^*-y^*\rangle\Big)\\
&=\tfrac{2-n}{n}\langle x,x^*\rangle+\inf_{y^*\in X^*}
\tfrac{2(n-1)}{n}\Big(F_A(x,y^*)+F_B(x,x^*-y^*)\Big)\\
&=\tfrac{2-n}{n}\langle x,x^*\rangle+\tfrac{2(n-1)}{n} F_A\Box_2F_{B}(x,x^*)\\
&=\tfrac{2-n}{n}\langle x,x^*\rangle+\tfrac{2(n-1)}{n} F_{A+B}(x,x^*),
\quad(\text{by \cite[Theorem~5.10]{BWY3}})\\
&=F_{A+B,\,n}(x,x^*)\quad(\text{by Theorem~\ref{Festival:2}}).
\end{align*}
Similarly, using \eqref{New:2}, we have $F_{A+B,\, \infty}= F_{A,\,\infty}\Box_2F_{B,\,\infty}$.
\end{proof}
\begin{remark}Theorem~\ref{FS6} generalizes
 \cite[Theorem~5.4]{BBW}.\end{remark}


\end{document}